\documentclass{amsart}[14pt,a4paper]
\usepackage{geometry}
\usepackage{amsfonts}
\usepackage{hyperref}
\usepackage{amsmath}
\usepackage{galois}
\DeclareMathSizes{10}{10}{7}{5}
\usepackage{mathtools}
\numberwithin{equation}{section}

\usepackage{enumerate}
\usepackage{amssymb}
\usepackage{mathrsfs}
\usepackage{color}
\newtheorem{theorem}{Theorem}[section]
\newtheorem{lemma}{Lemma}[section]
\newtheorem{remark}{Remark}

\newtheorem{corollary}{Corollary}[section]

\newtheorem{definition}{Definition}[section]

\newcommand{\pr}{\mbox{P}}
\newcommand{\ds}{\displaystyle}
\newcommand{\ex}{\mbox{E}}
\newcommand{\va}{\mbox{V}}

\newcommand{\wh}{\widehat}

\begin{document}
\title[\textbf{Inequalities Concerning Maximum Modulus and Zeros}]{
\textbf{Inequalities Concerning Maximum Modulus and Zeros of Random Entire Functions}\protect\footnotemark[2]}
\author{Hui Li}
\author{Jun Wang}
\author{Xiao Yao}
\author{Zhuan Ye}
\address{School of Science, Beijing University of Posts and Telecommunications, Beijing, 100876, China, and Department of Mathematics and Statistics, University of North Carolina, Wilmington, NC, 28403, USA.}
\address{School of Mathematical Sciences, Fudan University, Shanghai, 200433, China.}
\address{School of Mathematical Sciences and LPMC, Nankai University, Tianjin, 300071, China.}
\address{Department of Mathematics and Statistics, University of North Carolina, Wilmington, NC, 28403, USA.}
\renewcommand{\thefootnote}{\fnsymbol{footnote}}
\footnotetext[3]{Corresponding author: Hui Li.}
\maketitle
\begin{abstract}
Let $f_\omega(z)=\sum\limits_{j=0}^{\infty}\chi_j(\omega) a_j z^j$ be a random entire function, where  $\chi_j(\omega)$ are independent and identically distributed random variables defined on a probability space $(\Omega, \mathcal{F}, \mu)$. In this paper, we first define a family of random entire functions, which includes Gaussian, Rademacher, Steinhaus entire functions. Then, we prove that, for almost all functions in the family and
for any constant $C>1$, there exist a constant $r_0=r_0(\omega)$ and a set $E\subset [e, \infty)$ of finite logarithmic measure such that, for $r>r_0$ and
$r\notin E$,
$$
|\log M(r, f)- N(r,0, f_\omega)|\le (C/A)^{\frac1{B}}\log^{\frac1{B}}\log M(r,f)  +\log\log M(r, f), \qquad a.s.
$$
where $A, B$ are constants, $M(r, f)$ is the maximum modulus, and $N(r, 0, f)$ is the weighted counting-zero function of $f$. As a by-product of our main results, we prove Nevanlinna's second main theorem for random entire functions.
Thus, the characteristic function of almost all functions in the family is bounded above by a weighed counting function, rather than by two weighted counting functions in the classical Nevanlinna theory. For instance, we show that, for almost all Gaussian entire functions $f_\omega$ and for any $\epsilon>0$, there is $r_0$ such that, for $r>r_0$,
$$
T(r, f) \le N(r,0, f_\omega)+(\frac12+\epsilon) \log T(r, f).
$$
\end{abstract}

\section{Introduction}
Let $f$ be a transcendental entire function of the form
\begin{align}\label{def0}
f(z)=\sum_{j=0}^{\infty} a_j z^j,
\end{align}
where $z, a_j \in \mathbb{C}$.

Let $(\Omega,\, \mathcal{F}, \, \mu)$ be a probability space, where $\mathcal{F}$ is a $\sigma$-algebra of subset of $\Omega$ and $\mu$ is a probability measure on $(\Omega,\, \mathcal{F})$.
Along with the function (\ref{def0}), we consider the random functions on the probability space $(\Omega,\, \mathcal{F}, \, \mu)$ as follows:
\begin{equation}\label{def1}
f_\omega(z)=\sum_{j=0}^{\infty}\chi_j(\omega) a_j z^j,
\end{equation}
where $z, a_j \in \mathbb{C}$, $\omega \in \Omega$, $\chi_j(\omega)$ ($j=0, 1, 2, \cdots$) are independent and identically distributed complex-valued random variables. Further, we assume that the expectation and variance of $\chi_j$ are zero and one, respectively, for $j=0, 1, 2, \cdots.$
It is clear that $f_\omega(z)$ is an entire function for almost all $\omega \in \Omega$ (see: \cite{Kahane85}).

In general, ones consider three cases regarding $\chi_j(\omega)$.
{\bf Case one:} $\chi_j$ ($j=0, 1, \cdots$) are complex-valued Gaussian random variables with standard Gaussian distribution, and we call such $f_\omega(z)$ as {\bf Gaussian entire functions}; {\bf Case two:} $\chi_j$ ($j=0, 1, \cdots$) are Rademacher random variables, which take the values $\pm1$ with probability 1/2 each, and we call such $f_\omega(z)$ as {\bf Rademacher entire functions}; {\bf Case three:} $\chi_j=e^{2\pi i\theta_j}$ ($j=0, 1, \cdots$) are Steinhaus random variables, where $\theta_j$ ($j=0, 1, \cdots$) are independent real-valued random variables with uniform distribution in the interval [0,1], and we call such $f_\omega(z)$ as {\bf Steinhaus entire functions}.

The study of random polynomial was initiated by Bloch and P\'olya in 1932. Since then, there are a lot of publications on random polynomials. Moreover, the research on random transcendental entire functions, especially, on Gaussian, Rademacher and Steinhaus entire functions, has drawn a lot of attention, too (e.g. \cite{buGlSo2019, hoKrPeVi09, kabZap14, maFi10, MF12, NNS14, nns2016, sodin05,  SC00}).
Recently, Nazarov, Nishry and Sodin \cite{NNS14, nns2016} made a breakthrough on the logarithmic integrability of Rademacher Fourier series and obtained several important results on the distribution of zeros of Rademacher entire functions. Their results extended earlier work of Littlewood and Offord \cite{liOf45, LO48}.
Also,  in 1982, Murai \cite{mu82} proved the Nevanlinna defect identity for Rademacher entire functions. In 2000, Sun and Liu \cite{sun90} obtained the
Nevanlinna defect identity for $f(z)+X(\omega)g(z)$ (where $f ,g$ are entire, $g$ is a small function of $f$ and $X(\omega)$ is a non-degenerated complex-valued random variable). Later,
Mohola and Filevych \cite{maFi10, MF12} obtained  Nevanlinna's second main theorem for Steinhaus entire functions.

In this paper, we first define a family $\mathcal{Y}$ of random entire functions, which includes Gaussian, Rademacher and Steinhaus entire functions.
Thus, we can deal with these three classes of famous random entire functions all together.
Then, we prove several inequalities concerning the maximum modulus $M(r, f)$, $\sigma(r, f)$ and the weighted counting function $N(r, a, f_\omega)$ for the random entire functions in the family $\mathcal{Y}$. These inequalities show that the   counting-zero function of almost all randomly perturbed function $f_\omega$ is close to the maximum modulus of $f$, up to an error term. We are also carefully to treat the error terms in these inequalities.  Our Lemma \ref{lem2} verifies that the family $\mathcal{Y}$ includes
Gaussian, Rademacher and Steinhaus entire functions. The ingredients in our proofs involve the techniques used by Nazarov-Nishry-Sodin, Mohola-Filevych, Offord. As a by-product of our results, we also establish Nevanlinna's second main theorems for the random entire functions with a careful treatment of its error term.
Thus, we obtain that the characteristic function of almost all functions in the family is bounded above by a weighed counting function, rather than by two weighted counting functions in the classical Nevanlinna theory.

The paper is organized as follows. We devote Section \ref{S2} to giving some preliminaries and previous research results. In Section \ref{S3}, we state our main results and  Nevanlinna's second main theorems for random entire functions. In Section \ref{S4}, we give some lemmas which are needed in the proofs of our results, where Lemma \ref{lem2} is one of the key lemmas in the section. In Section \ref{S5}, we first prove Theorem \ref{thm1}, with which, then, we prove a lemma that has its own interests and is needed in the proof of Theorem \ref{newthm}. All corollaries are proved in this section, too.
\vskip.1in

\section{Preliminaries}\label{S2}

Let $X$ be a complex-valued random variable.  We denote the expected value and the variance of $X$ by $\ex(X)$ and $\va(X)$, respectively. In particular, if $X$ is a standard complex-valued Gaussian random variable (its the probability density function is $\ds e^{-|z|^2}/\pi$ with respect to Lebesgue measure $m$ in the complex plane), or a Rademacher random variable, or a Steinhaus random variable, then $\ex(X)=0$ and $\va(X)=\ex(|X|^2)=1$. We also denote the probability of an event $A$ by $\pr(A)$.

We say that almost all functions defined in (\ref{def1}) have a certain property almost surely (a.s.) if there is a set $F\subset \Omega$ such that $\mu(F)=0$ and the functions with $\omega \in \Omega \setminus F$ possessing the said property.

Define
$$
\sigma(r, f_\omega)=\left(\sum_{j=0}^{\infty}|a_j\chi_j(\omega)|^2r^{2j}\right)^{1/2}=\left(\int_0^{2\pi}|f_{\omega}(re^{i\theta})|^2\frac{d\theta}{2\pi}\right)^{1/2}
$$
and $ \sigma(r, f)=(\sum_{j=0}^{\infty}|a_j|^2r^{2j})^{1/2}$.
Further, if $\ex(\chi_j)=0$ and $\va(\chi_j)=1$, then
$$
\sigma^2(r, f)=\ex(|f_\omega(re^{i\theta})|^2)=\sum_{j=0}^{\infty}|a_j|^2r^{2j}.
$$
Set
\begin{eqnarray}\label{def2}
\hat{f_\omega}(re^{i \theta})\stackrel{def}{=}\frac{f_\omega(re^{i\theta})}{\sigma(r, f)}=\sum_{j=0}^{\infty}\chi_j(\omega) \frac{a_jr^j}{\sigma(r,f)}e^{ij\theta}\stackrel{def}{=}\sum_{j=0}^{\infty}\chi_j(\omega) \wh{a_j}(r)e^{ij\theta},
\end{eqnarray}
where $\ds \sum_{j=0}^{\infty}|\wh{a_j}(r)|^2=1$ for all $r$.
Let
$$
\ds X_r= \frac{1}{2\pi}\int_{0}^{2\pi}|\log|\hat{f_\omega}(re^{i\theta})|| d\theta, \qquad \mbox{for} \ r \in \mathbb{R}^+.
$$
\begin{definition}
Let $f$ and $ f_\omega$ be defined as in (\ref{def0}) and (\ref{def1}), respectively. Then the random entire function $f_\omega$ is in the {\bf family $\mathcal{Y}$} if and only if $f_\omega $ satisfies Condition $Y$, i.e.,
there are two positive constants $A$ and $B$ such that
\begin{equation*}\label{condition}
\noindent \mbox{\bf Condition Y:} \qquad   \qquad \ex(\exp(A|X_r|^B))<+\infty, \qquad \mbox{for all} \ r \in \mathbb{R}^+.
\end{equation*}
\end{definition}

In Section \ref{S4}, we will prove that all Gaussian, Rademacher, and Steinhaus entire functions are in family $\mathcal{Y}$. Indeed, if $f_\omega$ is Gaussian, Rademacher or Steinhaus, then $f_\omega$ satisfies Condition $Y$ when we choose $A\in (0, 2)$ and $B=1$; $A$ is close to zero and $B=1/6$; $A\in (0,1)$ and $B=1$; respectively.

It is well-known that if $\chi_j$ ($j=0, 1, 2, \cdots$) are standard complex-valued Gaussian random variables, then $\ex(X_r)$ is a positive constant. Therefore, for any Gaussian entire function $f_\omega$,
$$
\sup_{r>0}\ex(|N(r, 0, f_\omega)-\log\sigma(r, f)|)\le C,
$$
where $C$ is a constant.

In 2010 and 2012, Mahola and Filevych proved the following result, which can be regarded as a version of Nevanlinna's second main theorem.

\begin{theorem}[\rm{\cite{maFi10, MF12}}, Theorem 1] \label{MF}
Let $f$ be an entire function as defined in (\ref{def0}) and let $f_\omega(z)$ be a Steinhaus or a Gaussian entire function on $(\Omega,\, \mathcal{F}, \, \mu)$ of the form (\ref{def1}).
Then, there is a set $E$ of finite logarithmic measure on $(0, \infty)$ such that for every $a\in \mathbb{C}$, the inequality
$$
\log \sigma(r, f) \le  N(r,a, f_\omega) + C_1\log\log  \sigma(r,f) +O(1) \qquad a.s. \qquad (r\ge r_1(\omega), r\not\in E),
$$
holds, where $C_1>0$ is an absolute constant.
\end{theorem}
\begin{remark}
Indeed, in 2010, Mohola and Filevych \cite{maFi10} proved a similar inequality to that in Theorem \ref{MF} for the Steinhaus entire functions. In 2012, they proved Theorem \ref{MF} and other interesting results in \cite{MF12} for the Steinhaus entire functions. Further, in 2012, P. V. Filevych stated that the inequality in Theorem \ref{MF} is also true for the Gaussian entire functions.
Recently, P. V. Filevych told one of the authors that although the proof of the statement has not been published, it is essentially a repetition of the considerations from Mahola's Ph.D. dissertation  \cite{Mahola13}.
\end{remark}

Recently, Nazarov, Nishry and Sodin porved

\begin{theorem}[\rm{\cite{nns2016}}, Theorem 1.1] \label{nnsthm}
Let $f_\omega$ be a Rademacher entire function. There exists a set $E\subset [1, \infty)$ (depending on $|a_k|$ only) of finite logarithmic length such that
\begin{enumerate}[\rm(i)]
  \item for almost every $\omega\in \Omega$, there exists $r_0(\omega) \in [1, \infty)$ such that for every $r\in [r_0(\omega), \infty)\setminus E$ and every
  $\gamma >1/2$,
  $$
  |n(r, 0, f_\omega)-r\frac{d}{dr}\log\sigma(r, f)|\le C(\gamma)(r\frac{d}{dr}\log\sigma(r, f))^{\gamma};
  $$
  \item for every $r\in [1, \infty)\setminus E$ and every
$\gamma >1/2$,
$$
\ex|n(r, 0, f_\omega)-r\frac{d}{dr}\log\sigma(r, f)|\le C(\gamma)(r\frac{d}{dr}\log\sigma(r, f))^{\gamma}.
$$
\end{enumerate}
\end{theorem}

For the reader's convenience,
we recall some standard notations in function theory and state some important theorems in Nevanlinna theory for random entire functions $f_\omega$.
These notations and theorems will be used to prove new theorems in Nevanlinna theory as corollaries of our main results for random entire functions.

For any $\omega \in \Omega$, we define
$$
m(r, f_\omega)=\frac1{2\pi}\int_0^{2\pi}\log^+|f_\omega(re^{it})| \, dt
$$
and
$$
N(r, 0,  f_\omega)=\int_0^{r}\frac{n(t, 0, f_\omega)-n(0, 0, f_\omega)}t dt +n(0, 0, f_\omega)\log r,
$$
where $n(t, 0, f_\omega)$ is the number of zeros of $f_\omega(z)$ in the disk $D(0, t)$. We denote the Nevanlinna characteristic function of $f_\omega(z)$ by
$$
T(r, f_\omega)=m(r, f_\omega)+N(r, \infty, f_\omega)
$$
and the maximum modulus of
$f_\omega$ by $$\ds M(r, f_\omega)=\max_{|z|=r}|f_\omega(z)|.$$
As usual, for any $a\in \mathbb{C}$, we denote $N(r, 0, f_\omega -a)$ and $ T(r, 0, f_\omega-a) $ by $N(r, a, f_\omega)$ and $ T(r, a, f_\omega)$, respectively.

\begin{theorem}[{\bf Jensen-Poisson Formula}, e.g.   \cite{Hayman64,chYe01}]
If $f_\omega(z)$ is a random entire function, then
$$
\log |c_{f_\omega}(0)| + N(r, 0, f_\omega)=\frac1{2\pi}\int_0^{2\pi}\log |f_\omega(re^{is})|\, ds,
$$
where $c_{f_\omega}(0)$ is the first non-zero coefficient of Taylor series of $f_\omega(z)$.
\end{theorem}

The Jensen-Poisson formula also implies the so-called Nevanlinna's first main theorem.
\begin{theorem}[{\bf First Main Theorem}, e.g. \cite{Hayman64,chYe01}]\label{fmt}
Let $f_\omega(z)$ be a random entire function and $a\in \mathbb{C}$. Then
$$
 T(r, a, f_\omega)=T(r, f_\omega)-\log |c_{f_\omega}(0)| + \epsilon(a, r),
 $$
where $\ds |\epsilon(a, r)| \le \log^+|a| +\log 2.$
\end{theorem}

There are many versions of the second main theorem in Nevanlinna theory. Here, when $f_\omega$ is an entire function, we use the one with a better error term.
\begin{theorem}[{\bf Second Main Theorem}, e.g. \cite{Hayman64, chYe01}] \label{smt}
Let $f_\omega(z)$ be a random entire function and let $d_j$ ($j=1, 2.$) be two distinct complex numbers. Then
$$
T(r, f_\omega)\le N(r, d_1, f_\omega)+N(r, d_2, f_\omega)+ S(r, f_\omega),
$$
for all large $r$ outside a set $F$ of finite Lebesgue measure, where the error term
$$
S(r, f_\omega)\le \log T(r, f_\omega) +2\log \log T(r, f_\omega) +O_\omega(1).
$$
\end{theorem}

It is known (e.g. \cite{chYe01, ye95}) that the coefficient $1$ in the front of $\log T(r, f_\omega)$ in the inequality is the best and, clearly, the term $O_\omega(1)$ depends on $c_{f_\omega}(0)$ and $ d_j$.

Note that all above definitions and theorems also holds for the entire function $f$ defined in (\ref{def0}).

In the sequel, the values of constants, such as $C, C_1, r_0, r_1$, may be different in the each appearance of these constants.
\vskip.1in

\section{Our Results}\label{S3}

In this section, we state several inequalities concerning the maximum modulus $M(r, f)$, $\sigma(r, f)$ and the weighted counting function $N(r, 0, f_\omega)$ for the random entire functions in the family $\mathcal{Y}$ with careful treamtment of their error terms. A relationship between $\log \sigma(r, f_\omega)$ and
$\log \sigma(r, f)$ is stated and proved in Section $5$.

\begin{theorem}\label{thm1}
If $f_\omega \in \mathcal{Y}$, then,
for any constant $C>1$, there exists a constant $r_0=r_0(\omega)$ such that, for $r>r_0$,
$$
|\log \sigma(r, f)- N(r,0, f_\omega)|\le (C/A)^{\frac1{B}}\log^{\frac1{B}}\log \sigma(r,f),  \qquad a.s.
$$
where constants $A, B$ are from Condition $Y$.
\end{theorem}

\begin{remark}
Theorem \ref{thm1} tells us that the number of zeros of almost all $f_\omega$ can be controlled from above and below by $\log \sigma(r, f)$ and an error term, which are independent of $\omega$.
\end{remark}

Sometimes, it is easier for one to calculate $M(r, f)$ than $\sigma(r,f)$. By Lemma \ref{maxlem}, we obtain

\begin{corollary}\label{newcor3}
If $f_\omega \in \mathcal{Y}$, then,
for any constant $C>1$, there are a constant $r_0=r_0(\omega)$ and a set $E\subset [e, \infty)$ of finite logarithmic measure such that, for $r>r_0$ and
$r\notin E$,
$$
|\log M(r, f)- N(r,0, f_\omega)|\le (C/A)^{\frac1{B}}\log^{\frac1{B}}\log M(r,f)  +\log\log M(r, f), \qquad a.s.
$$
where constants $A, B$ are from Condition $Y$.
\end{corollary}
\noindent
{\bf Example.} Let $f(z)=e^z$ and its random perturbation function $f_\omega$ in the family $\mathcal{Y}$. Then the corollary tells us that, for almost all $f_\omega$, its weighted counting-zero function in the disk $D(0, r)$ is close to $r$ although $e^z$ does not take the value zero at all.

Now we state Nevanlinna's second main theorem (involving the weighted counting-zero function only) for random entire functions as corollaries of above results.

Theorem \ref{thm1} and Lemma \ref{lem3} easily imply the next corollary.
\begin{corollary}\label{cor1}
If $f_\omega \in \mathcal{Y}$,
then, for any constant $C>1$, there exists a constant $r_0=r_0(\omega)$ such that, for $r>r_0$,
$$
T(r, f)\leq N(r,0, f_\omega)+(C/A)^{\frac1{B}}\log^{\frac1{B}} T(r,f), \qquad a.s.
$$
and
$$
T(r, f_\omega)\leq N(r,0, f_\omega)+(C/A)^{\frac1{B}}\log^{\frac1{B}} T(r,f_\omega), \qquad a.s.
$$
where constants $A, B$ are from Condition $Y$.
\end{corollary}

When $f_\omega$ is a Gaussian, or Rademacher, or Steinhaus entire function, we have the following corollary.

\begin{corollary}\label{cor2}
Let $f$ and $ f_\omega$ be defined as in (\ref{def0}) and (\ref{def1}), respectively. Then, for any $\epsilon>0$,  there exists $r_0=r_0(\omega, \epsilon)$ such that, for $r>r_0$,
\begin{enumerate}[\rm(i)]
  \item if $f_\omega(z)$ is a Gaussian entire function, then
  $$
  T(r, f) \leq N(r,0, f_\omega)+\frac{1+\epsilon}{2} \log T(r,f) \qquad a.s.
  $$
   \item if $f_\omega(z)$ is a Rademacher entire function, then
  $$
  T(r, f) \leq N(r,0, f_\omega)+ \Big(\Big(\frac{eC_0}6\Big)^6 +\epsilon \Big) \log^6 T(r,f)\qquad a.s,
  $$
  where $C_0$ is from Lemma \ref{Sodin}.
   \item if $f_\omega(z)$ is a Steinhaus entire function, then
  $$
  T(r, f) \leq N(r,0, f_\omega)+(1+\epsilon) \log T(r,f) \qquad a.s.
  $$
\end{enumerate}
\end{corollary}

Now we consider the case when $f_\omega$ takes any value $a\in \mathbb{C}$.

\begin{theorem}\label{newthm}
Let $f_\omega \in \mathcal{Y}$ and define
$$
 f^*_\omega(z)=zf'_\omega(z)=\sum_{j=1}^{\infty}ja_j\chi_jz^j.
 $$
If $f^*_\omega$ satisfies Condition $Y$, then, for any constant $C>1$, there exists a set $E$ of finite logarithmic measure such that, for every $a\in \mathbb{C}$, there is $r_1=r_1(\omega,a)$ such that, for $r>r_1$ and $r\not\in E$,
$$
|\log \sigma(r, f)- N(r,a, f_\omega)|\le (C/A)^{\frac1{B}}\log^{\frac1{B}}\log \sigma(r,f) +(1+o(1))\log\log\sigma(r, f), \qquad a.s.
$$
where constants $A, B$ are from Condition $Y$.
\end{theorem}
\begin{remark}
If $f$ is a Gaussian, or Rademarcher, or Steinhaus entire function, then it follows from Lemma \ref{lem2} that both $f_\omega$ and $f^*_\omega$ satisfy Condition $Y$.
\end{remark}

The following corollary is a straightforward consequence of the above theorem and Lemma \ref{maxlem}.

\begin{corollary}\label{newcor4}
Under the assumptions of Theorem \ref{newthm}, we have that, for any constant $C>1$, there exists a set $E$ of finite logarithmic measure such that, for every $a\in \mathbb{C}$, there is $r_1=r_1(\omega,a)$ such that, for $r>r_1$ and $r\not\in E$,
$$
|\log M(r, f)- N(r,a, f_\omega)|\le (C/A)^{\frac1{B}}\log^{\frac1{B}}\log M(r,f) +(2+o(1))\log\log M(r, f), \qquad a.s.
$$
where constants $A, B$ are from Condition $Y$.
\end{corollary}

When $f_\omega$ is Gaussian, Rademacher or Steinhuas, Theorem \ref{newthm} and Lemma \ref{lem2} give

\begin{corollary}\label{cor4}
Let $f$ and $ f_\omega$ be defined as in (\ref{def0}) and (\ref{def1}), respectively. 
Then, for any $\epsilon>0$, there exists a set $E$ of finite logarithmic measure such that, for every $a\in \mathbb{C}$, there exists $r_0=r_0(\omega, \epsilon, a)$ such that, for $r>r_0$ and $r\not\in E$, we have
\begin{enumerate}[\rm(i)]
  \item if $f_\omega(z)$ is a Gaussian entire function, then
  $$
  \log \sigma (r, f) \leq N(r,a, f_\omega)+(\frac32+\epsilon) \log \log\sigma (r, f) \qquad a.s.
  $$
  \item if $f_\omega(z)$ is a Rademacher entire function, then
  $$
  \log \sigma (r, f) \leq N(r,a, f_\omega)+ \Big(\Big(\frac{eC_0}6\Big)^6 +\epsilon \Big) \log^6 \log \sigma(r,f)\qquad a.s.,
  $$
  where $C_0$ is from Lemma \ref{Sodin}.
   \item if $f_\omega(z)$ is a Steinhaus entire function, then
  $$
  \log\sigma(r, f) \leq N(r,a, f_\omega)+(2+\epsilon) \log \log\sigma(r,f) \qquad a.s.
  $$
\end{enumerate}
\end{corollary}

\begin{remark} Corollary \ref{cor4} shows that the coefficients in the error term is $3/2 +\epsilon$ and $2+ \epsilon$ in Gaussian and Steinhaus cases, rather than a constant $C_1>0$ in Theorem \ref{MF}.
It's interesting to know whether these coefficients are the best possible coefficients in these error terms.
\end{remark}

The following is  Nevanlinna's second main theorem for random entire functions .
It verifies that the characteristic function for almost all random entire functions can be bounded above by one weighted counting function, rather than two weighted counting functions in the classical case (e.g. Theorem \ref{smt}).
The proof of the following corollary is a straightforward consequence of Theorem \ref{newthm} and Lemma \ref{lem3} as we have seen the proof of
Corollary \ref{cor1}.
\begin{corollary}\label{newcor}
If $f_\omega $ and $f^*_\omega$ satisfy Condition $Y$,
then, for any constant $C>1$, there exists a set $E$ of finite logarithmic measure such that, for every $a\in \mathbb{C}$,  there is $r_1=r_1(\omega,a)$ such that, for $r>r_1$ and $r\not\in E$,
$$
T(r, f)\leq N(r,a, f_\omega)+(C/A)^{\frac1{B}}\log^{\frac1{B}} T(r,f)+(1+o(1))\log T(r,f),\qquad a.s.
$$
and
$$
T(r, f_\omega)\leq N(r,a, f_\omega)+(C/A)^{\frac1{B}}\log^{\frac1{B}} T(r,f_\omega)+(1+o(1))\log T(r,f_\omega), \qquad a.s.
$$
where constants $A, B$ are from Condition $Y$.
\end{corollary}

\section{Lemmas}\label{S4}
One of the key lemmas in this section is Lemma \ref{lem2}. Lemma \ref{Offord2} is proved by using Offord's work in Lemma \ref{Offord}.
There is another lemma in Section \ref{S5} because the proof of Lemma \ref{lem6} needs Theorem \ref{thm1} whose proof is in Section \ref{S5}.


\begin{lemma}\label{lem1}
Let $X$ be a random variable and $\psi(x), \phi(x)$ be positive increasing differentiable functions on $(0, \infty)$. Then, for any $x>0$,
$$\pr(|X|\ge \psi(x))\le \frac{\ex(\phi(|X|))}{\phi(\psi(x))},
$$
provided that the probability density function and the expectation of $\phi(|X|))$ exist.
\end{lemma}

\begin{proof}
Let $g$ be the probability density function of $\phi(|X|)$. Then
\begin{align*}
\ex(\phi(|X|))&= \int_0^{\infty}tg(t)\, dt \ge \int_{\phi(\psi(x))}^{\infty}tg(t)\,dt
\ge \phi(\psi(x))\int_{\phi(\psi(x))}^{\infty}g(t)\,dt\\
&= \phi(\psi(x))\,\pr(\phi(|X|)\ge  \phi(\psi(x)))
= \phi(\psi(x))\,\pr(|X|\ge  \psi(x)).
\end{align*}
\end{proof}

\begin{lemma}[{\bf Log-integrability} \cite{NNS14}]\label{Sodin}
Let $f_\omega$ be a Rademarcher entire function. Then, for any $p\geq1$,
$$
\ex\left(\int_{0}^{2\pi}|\log|\hat{f_\omega}||^p \,\frac{d\theta}{2\pi}\right)\leq (C_0p)^{6p},
$$
where $C_0$ is an absolute constant.
\end{lemma}

\begin{lemma}[{\bf Offord} \cite{offord72}]\label{Offord}
Let $S=\sum\limits_{j=0}^{\infty}c_j$ is a convergent series of non-negative terms, then we can find $k$ such that
$$
\sum_{j=0}^{k}c_j\geq\frac1{4}, \qquad \sum_{j=k}^{\infty}c_j\geq\frac1{2}S,\qquad c_k\geq \frac1{4}(k+1)^{-2}S.
$$
Furthermore, if in addition $k\geq2$, then
$$
\int_0^{\infty}\Big|\prod_{j=0}^{\infty}J_0(\rho\sqrt{c_j})\Big|\, d\rho \leq C\left(\sum_{j=0}^{\infty}c_j\right)^{-1/2},
$$
where $C$ is an absolute constant, $\ds J_0(\rho)=\int_{0}^{1}\exp(i\zeta \cos2\pi u+i\eta \sin 2\pi u)\,du$ is a Bessel function of the first kind of order 0, $\rho=\sqrt{\zeta^2+\eta^2}$.
\end{lemma}

\begin{lemma}\label{Offord2}
Let $f_\omega(z)$ be a Steinhaus entire function on $(\Omega,\, \mathcal{F}, \, \mu)$ of the form (\ref{def1}) and let $\hat{f_\omega}$ be of the form (\ref{def2}). For $t>0$, $\phi\in[0,2\pi)$, set
$$A^*=\{\omega\in\Omega:|{\rm Re}(\hat{f_\omega})\cos\phi+{\rm Im}(\hat{f_\omega})\sin\phi|<t\}.$$
Then there is a constant $r_0$ such that, for $|z|=r>r_0$,
$$
\pr(A^*)\leq C t,
$$
where $C$ is an absolute constant.
\end{lemma}
\begin{proof}
Let
$$
\hat{f_\omega}(re^{i\theta})=\sum_{j=0}^{\infty}\wh{a_j}(r)e^{ij\theta}e^{2\pi i \theta_j(\omega)},
$$
where $\theta_j(\omega)$ are independent real-valued random variables with uniform distribution in the interval [0,1].

Denote the characteristic function of $\hat{f_\omega}$ by $\Psi_{\hat{f_\omega}}(\zeta+i \eta)=\ex e^{ i({\rm Re}(\hat{f_\omega}) \zeta+{\rm Im}(\hat{f_\omega})\eta)}$.
Since the characteristic function of $e^{2\pi i \theta_n(\omega)}$ is
$$
\int_0^1\exp[i(\zeta \cos(2\pi u)+\eta \sin(2\pi u))]\, du=J_0(\rho),
$$
which depends only on the modulus of the terms. Therefore,
$$
\ex e^{ i({\rm Re}(\hat{f_\omega}) \zeta+{\rm Im}(\hat{f_\omega})\eta)}=\prod\limits_{j=0}^{\infty}J_0(\rho|\wh{a_n}(r)|).
$$

We know
$$
\pr(A^*)=\frac{2}{\pi}\int_{0}^{\infty}
\frac{\sin(t\rho)}{\rho}{\rm Re} \prod\limits_{j=0}^{\infty}J_0(\rho|\wh{a_n}(r)|)\,d\rho,
$$
where $\rho=\sqrt{\zeta^2+\eta^2}$. For the convenience of readers, we add a detailed proof here.

Set $\xi=\rho\cos\phi, \eta=\rho \sin\phi$ and $\phi\in [0, 2\pi)$. Then
\begin{align*}
&\int_{0}^{\infty}\frac{\sin(t\rho)}{\rho}{\rm Re} \ex e^{i({\rm Re}(\hat{f_\omega}) \zeta+{\rm Im}(\hat{f_\omega})\eta)}d\rho\\
=&{\rm Re} \ex \Big(\int_{0}^{\infty}\frac{\sin(t\rho)}{\rho} e^{i({\rm Re}(\hat{f_\omega}) \zeta+{\rm Im}(\hat{f_\omega})\eta)}d\rho\Big)= \ex \Big(\int_{0}^{\infty}\frac{\sin(t\rho)}{\rho} \cos{({\rm Re}(\hat{f_\omega}) \zeta+{\rm Im}(\hat{f_\omega})\eta)}d\rho\Big)\\
=&\frac{1}{2}\ex \Big(\int_{0}^{\infty}(
\frac{\sin (\rho({\rm Re}(\hat{f_\omega})\cos\phi+{\rm Im}(\hat{f_\omega})\sin\phi+t))}{\rho}
+\frac{\sin (\rho(-{\rm Re}(\hat{f_\omega})\cos\phi-{\rm Im}(\hat{f_\omega})\sin\phi+t))}{\rho})d\rho\Big).
\end{align*}
We denote
\begin{align*}
&  Q(t,\phi) = \\
& \ \int_{0}^{\infty}(
\frac{\sin (\rho({\rm Re}(\hat{f_\omega})\cos\phi+{\rm Im}(\hat{f_\omega})\sin\phi+t))}{\rho}
+\frac{\sin (\rho(-{\rm Re}(\hat{f_\omega})\cos\phi-{\rm Im}(\hat{f_\omega})\sin\phi+t))}{\rho})d\rho,
\end{align*}
and recall that
$$
\int_{0}^{\infty}\frac{\sin(\alpha t)}{t}dt=\frac{\pi}{2}\operatorname{sign}(\alpha).
$$
Then
$$
Q(t,\phi)=\pi1_{A^*}(\omega),
$$
where $1_{A^*}(\omega)$ is indicator function.
Hence
$$
\int_{0}^{\infty}\frac{\sin(t\rho)}{\rho}{\rm Re} \ex e^{i({\rm Re}(\hat{f_\omega}) \zeta+{\rm Im}(\hat{f_\omega})\eta)}d\rho=\frac1{2}\ex(Q(t,\phi))=\frac{\pi}{2} \pr(A^*).
$$
Therefore,
$$
\pr(A^*)=\frac{2}{\pi}\int_{0}^{\infty}
\frac{\sin(t\rho)}{\rho}{\rm Re} \prod\limits_{j=0}^{\infty}J_0(\rho|\wh{a_j}(r)|)\,d\rho.
$$

Applying Lemma \ref{Offord} to series $\sum\limits_{j=0}^{\infty}|\wh{a_j}(r)|^2$, we can easily see that the condition ``$k\geq2$'' is satisfied for sufficiently large $r$, and then together with $\sum\limits_{j=0}^{\infty}|\wh{a_j}(r)|^2=1$, for any $t>0$, we have
\begin{align*}
\pr(A^*)\leq \frac{2t}{\pi}\int_{0}^{\infty}
\Big|\frac{\sin(t\rho)}{t\rho}\operatorname{Re} \prod\limits_{j=0}^{\infty}J_0(\rho|\wh{a_j}(r)|)\Big|\,d\rho \leq \frac{2t}{\pi} \int_0^{\infty}\Big|\prod_{j=0}^{\infty}J_0(\rho|\wh{a_j}(r)|)\Big|\, d\rho\leq Ct,
\end{align*}
where $C$ is an absolute constant.
Therefore, the lemma is proved.
\end{proof}

\begin{lemma}\label{lem2}
Let $f_\omega(z)\in \mathcal{Y}$ and let $\hat{f_\omega}(re^{i\theta})$ be defined in (\ref{def2}). Then,  for any positive constant $C$ and all $x>1$, there is a positive constant $C_1$ such that
\begin{equation}\label{lem2-1}
\pr\left(|X_r|\ge \left(\frac{C}{A}\log x\right)^{\frac1{B}}\right)\le \frac{C_1}{x^{C}}.
\end{equation}
In particular, we have the following:
\begin{enumerate}
  \item  If $\chi_j(\omega)$ are standard complex-valued Gaussian random variables, then, for any $\tau>0$,
      $$
      \pr\left(|X_r|\ge \frac{1+2\tau}{2}\log x\right)\le \frac{C_1}{x^{\frac{1+2\tau}{1+\tau}}}.
      $$
  \item
  If $\chi_j(\omega)$ are Rademacher random variables, then, for
  any $\tau>0$, $\ds \varepsilon\in(0,\frac6{eC_0})$ ($C_0$ is from Lemma \ref{Sodin}),
      $$
      \pr\left(|X_r|\ge \left(\frac{1+\tau}{\varepsilon}\right)^6\log^6 x\right)\le \frac{C_1}{x^{1+\tau}}.
      $$
  \item
If $\chi_j$ are Steinhaus random variables, then, for any $\tau>0$, there is a constant $r_0$, for $r>r_0$,
      $$
      \pr\big(|X_r|\ge (1+\tau)^2\log x\big)\le \frac{C_1}{x^{1+\tau}}.
      $$
\end{enumerate}
\end{lemma}

\begin{proof}
By Lemma \ref{lem1}, we obtain
$$
\pr\left(|X_r|\ge \left(\frac{C}{A}\log x\right)^{\frac1{B}}\right)\le \frac{\ex(\exp(A|X_r|^B))}{\exp(A(\frac{C}{A}\log x))}\stackrel{def}{=} \frac{C_1}{x^{C}}.
$$

Now we give the proof of (i).
Since $\chi_j$ are independent standard complex-valued Gaussian random variables, then
$$
 \ex(\hat{f_\omega}(re^{i\theta}))=0
\qquad \mbox{and} \qquad
\va(\hat{f_\omega}(re^{i\theta})) = \ex(|\hat{f_\omega}(re^{i\theta})|^2)=1.
$$
It follows that $\hat{f}_\omega(re^{i\theta})$ is a standard complex-valued Gaussian random variable. For any $x>0$,
$$\pr(|\log|\hat{f_\omega}(re^{i\theta})||)<x) = \pr(-x<\log|\hat{f_\omega}(re^{i\theta})|<x) =e^{-e^{-2x}}-e^{-e^{2x}}.\\
$$
Consequently, the probability density function of $\ds |\log|\hat{f_\omega}(re^{i\theta})||$ is $2e^{-e^{-2x}}e^{-2x}+2e^{-e^{2x}}e^{2x}$, for $x>0$ and is $0$, for $x\le 0$.
Thus, we have
\begin{align}\label{lem4.4-1} \nonumber
\ex(e^{\frac{2}{1+\tau}|X_r|})
&=\sum_{n=0}^{\infty}\frac{2^n}{n!(1+\tau)^n}\ex|X_r|^n\\ \nonumber
&=\sum_{n=0}^{\infty}\frac{2^n}{n!(1+\tau)^n}\ex\left( \frac{1}{2\pi}\int_{0}^{2\pi}|\log |\hat{f_\omega}(re^{i\theta})||d\theta \right)^n\\\nonumber
&\leq\sum_{n=0}^{\infty}\frac{2^n}{n!(1+\tau)^n}\left( \frac{1}{2\pi}\int_{0}^{2\pi}\ex|\log |\hat{f_\omega}(re^{i\theta})||^n d\theta \right)
=\ex(e^{\frac{2}{1+\tau}|\log|\hat{f_\omega}(re^{i\theta})||})\\
&=\int_0^{\infty}e^\frac{2x}{1+\tau}
(2e^{-e^{-2x}}e^{-2x}+2e^{-e^{2x}}e^{2x})\, dx\stackrel{def}{=}C_1<\infty,
\end{align}
where $C_1$ is a positive constant. It follows that
Gaussian entire functions are in the family $\mathcal{Y}$ by taking
$A=\frac2{1+\tau}$ and $B=1$. Set $C=\frac{1+2\tau}{1+\tau}$.
Then, by (\ref{lem2-1}) and for $x\ge 1$, we get
$$
\pr\left(|X_r|\ge \frac{1+2\tau}{2}\log x\right)\le \frac{C_1}{x^{\frac{1+2\tau}{1+\tau}}}.
$$
This completes the proof of (i).

Next, we prove (ii).
By Lemma \ref{Sodin}, we have, for any positive integer $n\ge 6$,
\begin{align*}
\ex|X_r|^{\frac{n}{6}}=\ex \left(\int_0^{2\pi}|\log|\hat{f_\omega}||\,\frac{d\theta}{2\pi}\right)^{\frac{n}{6}}\leq \ex\left(\int_0^{2\pi}|\log|\hat{f_\omega}||^{\frac{n}{6}}\,\frac{d\theta}{2\pi}\right)\leq \left(\frac{C_0n}{6}\right)^n,
\end{align*}
where $C_0$ is the constant from Lemma \ref{Sodin}. Thus, when $\epsilon \in
(0, \frac{6}{eC_0})$,
\begin{equation}\label{lem4.4-3}
C_1\stackrel{def}{=}\ex \left(\exp(\varepsilon |X_r|^{\frac1{6}}) \right)=\sum_{n=0}^{\infty}\frac{\ex|\epsilon X_r|^{\frac{n}{6}}}{n!}\leq \sum_{n=6}^{\infty}\left(\frac{C_0\varepsilon n}{6}\right)^n\frac1{n!}+O(1)<+\infty.
\end{equation}
Therefore, Rademacher entire functions satisfy the condition $Y$ by choosing $A=\varepsilon$ and $B=\frac1{6}$.  Using the inequality (\ref{lem2-1}) for $C=1+\tau$, we get
\begin{align*}
\pr\left(|X_r|\ge \left(\frac{1+\tau}{\varepsilon}\right)^6\log^6 x\right)&=\pr \big(\exp(\varepsilon |X_r|^{\frac1{6}})\geq x^{1+\tau} \big)\leq \frac{\ex \left(\exp(\varepsilon |X_r|^{\frac1{6}}) \right)}{x^{1+\tau}}=\frac{C_1}{x^{1+\tau}}.
\end{align*}

Now, we prove (iii).

For any non-negative integer $j$ and any $\varphi\in[0,2\pi)$, set
$$ b_j=\wh{a_j}(r)\chi_je^{ij\theta}=\wh{a_j}(r)e^{ij\theta}e^{i2\pi \theta_j} \qquad \mbox{and} \qquad
B_j=\mbox{Re}(b_j)\cos \varphi +\mbox{Im}(b_j)\sin \varphi.
$$
Thus, $\hat{f_\omega}(re^{i\theta})=\sum\limits_{j=0}^\infty b_j$. Further,  we deduce that
$
B_j=u\cos(2\pi\theta_j)+v\sin(2\pi \theta_j),
$
where 
\begin{align*}
u ={\rm Re}(\wh{a_j}(r))\cos (j\theta) \cos \varphi & + {\rm Re}(\wh{a_j}(r))\sin (j\theta) \sin \varphi \\
& -{\rm Im}(\wh{a_j}(r))\sin (j\theta) \cos \varphi+{\rm Im}(\wh{a_j}(r))\cos (j\theta) \sin \varphi,\\
v  = {\rm Re}(\wh{a_j}(r))\cos (j\theta) \sin \varphi & -{\rm Re}(\wh{a_j}(r))\sin (j\theta) \cos \varphi \\
& -{\rm Im}(\wh{a_j}(r))\cos (j\theta) \cos \varphi-{\rm Im}(\wh{a_j}(r))\sin (j\theta) \sin \varphi
\end{align*}
and
$u^2+v^2=|b_j|^2$. Therefore, as we have seen from previous proofs, the characteristic function of $B_j$ is
$$
\int_0^1\exp[i(u \cos(2\pi t)+v\sin(2\pi t))]\, dt=J_0(\rho),
$$
where $\rho =(u^2+v^2)^{1/2}.$
Similarly, we obtain that, for $\varphi\in[0,2\pi)$,
$$
\ex(|\log|\mbox{Re} (\hat{f_\omega}(re^{i\theta}))\cos \varphi+\mbox{Im}(\hat{f_\omega}(re^{i\theta}))\sin\varphi||^n)
$$
is independent of $\varphi$ for any non-negative integer $n$.

Since
\begin{align*}
\mbox{Re} (\hat{f_\omega}(re^{i\theta}))\cos \varphi+\mbox{Im}(\hat{f_\omega}(re^{i\theta}))\sin\varphi &  =
\sqrt{\mbox{Re}^2(\hat{f_\omega}(re^{i\theta}))+\mbox{Im}^2 (\hat{f_\omega}(re^{i\theta}))}\sin(\varphi+\varphi_0)\\
& =|\hat{f_\omega}(re^{i\theta})|\sin(\varphi+\varphi_0),
\end{align*}
(where $\sin\varphi_0=\mbox{Re} (\hat{f_\omega}(re^{i\theta}))/|\hat{f_\omega}(re^{i\theta})|$, and $\cos\varphi_0=\mbox{Im} (\hat{f_\omega}(re^{i\theta}))/|\hat{f_\omega}(re^{i\theta})|$), so,
\begin{align*}
\log|\hat{f_\omega}(re^{i\theta})|&=\int_{0}^{2\pi}\log|\mbox{Re} (\hat{f_\omega}(re^{i\theta}))\cos \varphi+\mbox{Im} (\hat{f_\omega}(re^{i\theta}))\sin\varphi|\, \frac{d\varphi}{2\pi}-\int_{0}^{2\pi}\log|\sin \varphi|\, \frac{d\varphi}{2\pi}\\
        &=\int_{0}^{2\pi}\log|\mbox{Re} (\hat{f_\omega}(re^{i\theta}))\cos \varphi+\mbox{Im} (\hat{f_\omega}(re^{i\theta}))\sin\varphi|\, \frac{d\varphi}{2\pi}+\log2.
\end{align*}
Then, this together with Jensen inequality gives, for any $A>0$,
\begin{align*}
\ex(e^{A|X_r|})&=\sum_{n=0}^{\infty}\frac{A^n}{n!}\ex|X_r|^n
=\sum_{n=0}^{\infty}\frac{A^n}{n!}\ex\left( \frac{1}{2\pi}\int_{0}^{2\pi}|\log |\hat{f_\omega}(re^{i\theta})||d\theta \right)^n\\
&\leq\sum_{n=0}^{\infty}\frac{A^n}{n!}\left( \frac{1}{2\pi}\int_{0}^{2\pi}\ex|\log |\hat{f_\omega}(re^{i\theta})||^n d\theta \right)\\
&=\sum_{n=0}^{\infty}\frac{A^n}{n!}\left(\ex|\log |\hat{f_\omega}(re^{i\theta})||^n  \right)
= \ex(e^{A|\log|\hat{f_\omega}(re^{i\theta})||})\\
&\leq 2^A\ex(e^{A|\frac1{2\pi}\int_0^{2\pi}\log|{\rm{Re} }(\hat{f_\omega}(re^{i\theta}))\cos \varphi+{\rm{Im}}(\hat{f_\omega}(re^{i\theta}))\sin\varphi|d\varphi|})\\
& \le 2^A \sum_{n=0}^{\infty}\frac{A^n}{n!} \frac1{2\pi}\int_0^{2\pi}\ex(|\log|{\rm{Re}}(\hat{f_\omega}(re^{i\theta}))\cos \varphi+{\rm{Im}}(\hat{f_\omega}(re^{i\theta}))\sin\varphi||^nd\varphi )
\\
&\leq C_0\ex(e^{A|\log|{\rm{Re} }(\hat{f_\omega}(re^{i\theta}))\cos \varphi+{\rm{Im}}(\hat{f_\omega}(re^{i\theta}))\sin\varphi||}),
\end{align*}
where $C_0$ is a positive constant.
Let
$$
V_\omega(\theta,\varphi,r)=|{\rm{Re}} (\hat{f_\omega}(re^{i\theta}))\cos \varphi+{\rm{Im}} (\hat{f_\omega}(re^{i\theta}))\sin\varphi|,
$$
and
$$
Y_1=\{\omega\in\Omega:V_\omega(\theta,\varphi,r)>1\},  \quad Y_2=\{\omega\in\Omega:V_\omega(\theta,\varphi,r)\leq1\}.
$$
Since
$$
\ex(|\hat{f_\omega}(re^{i\theta})|^2)=\ex(\sum_{j=0}^{\infty}\sum_{k=0}^{\infty}e^{2\pi i \theta_j(\omega)}\overline{e^{2\pi i \theta_k(\omega)}} \wh{a_j}(r)\overline{\wh{a_k}(r)}e^{ij\theta-ik\theta})=\sum_{j=0}^{\infty} \ex(|e^{2\pi i \theta_j(\omega)}|^2)|\wh{a_j}(r)|^2=1,
$$
then
\begin{align*}
\int_{Y_1}e^{A\log V_\omega(\theta,\varphi,r)}\, d\pr(\omega)&=\int_{Y_1}|\mbox{Re} (\hat{f_\omega}(re^{i\theta}))\cos \varphi+\mbox{Im} (\hat{f_\omega}(re^{i\theta}))\sin\varphi|^{A} \, d\pr(\omega)\\
&\leq \int_{Y_1}|\hat{f_\omega}(re^{i\theta})|^{A} \, d\pr(\omega)
\leq \left( \int_{\Omega}|\hat{f_\omega}(re^{i\theta})|^2 \, d\pr(\omega)\right)^{A/2}\\
&=\left(\ex(|\hat{f_\omega}(re^{i\theta})|^2)\right)^{A/2}=1.
\end{align*}
On the set $Y_2$, by Lemma \ref{Offord2}, there is a constant $r_0$, for $r>r_0$, we have
\begin{align*}
\int_{Y_2} e^{A|\log V_\omega(\theta,\varphi,r)|}\,& d\pr(\omega)=\int_{0}^{\infty} \pr(\{\omega\in Y_2: e^{A|\log V_\omega(\theta,\varphi,r)|}\geq \lambda\})\, d\lambda\\
&=\int_{0}^{\infty} \pr(\{\omega\in Y_2: e^{-A\log|{\rm Re }(\hat{f_\omega}(re^{i\theta}))\cos \varphi+{\rm Im} (\hat{f_\omega}(re^{i\theta}))\sin\varphi|}\geq \lambda\})\, d\lambda\\
&=\int_{0}^{\infty} \pr(\{\omega\in Y_2: |{\rm Re} (\hat{f_\omega}(re^{i\theta}))\cos \varphi+{\rm Im} (\hat{f_\omega}(re^{i\theta}))\sin\varphi|\leq (1/\lambda)^{1/A}\})\, d\lambda\\
&\leq1+\int_{1}^{\infty} \pr(\{\omega\in \Omega: |{\rm Re} (\hat{f_\omega}(re^{i\theta}))\cos \varphi+{\rm Im} (\hat{f_\omega}(re^{i\theta}))\sin\varphi|\leq (1/\lambda)^{1/A}\})\, d\lambda\\
&\leq 1+ C\int_{1}^{\infty}\frac{d\lambda}{\lambda^{1/A}}.
\end{align*}
Thus, when $0<A<1$, we have
$$
\int_{Y_2} e^{A |\log V_\omega(\theta,\varphi,r)|}\, d\pr(\omega)\leq C\int_{1}^{\infty}\frac{d\lambda}{\lambda^{1/A}} +O(1)<+\infty.
$$
Therefore, set $A=1/(1+\tau)$ as before, we obtain
\begin{align}\label{lem4.4-2}
\ex(e^{\frac1{1+\tau}|X_r|})\le C_0\int_{Y_1}e^{A\log V_\omega(\theta,\varphi,r)}\, d\pr(\omega)+
C_0\int_{Y_2} e^{A|\log V_\omega(\theta,\varphi,r)|}\, d\pr(\omega)
=C_1<+\infty.
\end{align}
It follows that
$$
\pr\big(|X_r|\ge (1+\tau)^2\log x\big)\le \frac{C_1}{x^{1+\tau}}.
$$
This completes the proof of the lemma.
\end{proof}

\begin{lemma}\label{lem3}
Let $f$ and $f_\omega$ be entire functions of the forms (\ref{def0}) and (\ref{def1}), respectively. Then there is a constant $r_1>0$ such that,  for $r>r_1$, 
$$
T(r,f)\leq \log\sigma(r,f)+\frac12\log2 \qquad \mbox{and} \qquad
T(r,f_\omega)\leq \log\sigma(r,f_\omega)+\frac12\log2.
$$
\end{lemma}

\begin{proof}
By Parseval equality and Jensen inequality, we obtain
$$T(r,f_\omega)=\frac1{2\pi}\int_{0}^{2\pi}\log^+|f_\omega(re^{i\theta})|d\theta
\leq \frac1{4\pi}\int_{0}^{2\pi}\log(|f_\omega(re^{i\theta})|^2+1)d\theta\le \log \sigma(r,f_\omega)+\frac12\log2.
$$
The other inequality in the lemma can be proved in the same manner. 
\end{proof}

\begin{lemma}[{\bf Plane Growth Lemma}, e.g. \cite{chYe01}]\label{Borel}
Let $F(r)$ be a positive, nondecreasing continuous function satisfying $F(r)\geq e$ for $e<r_0<r<\infty$. Let $\psi(r) \ge 1$ be a real-valued, continuous, non-decreasing function on the interval $[e, \infty)$ and $\ds \int_e^{\infty}\frac{dr}{r\psi(r)}<\infty$. Let $\phi(r)$ be a positive, nondecreasing function defined for $r_0\leq r<\infty$. Set $R=r+\phi(r)/\psi(F(r))$.  Then there exists a closed set $E\subset [r_0,\infty)$ with $\ds \int_E\frac{dr}{\phi(r)} <\infty$ such that for all $r>r_0$, $r\not\in E$, we have
$$
\log F(R)<\log F(r)+1,
$$
and
$$
\log\frac{R}{r(R-r)}\le \log \frac{\psi(F(r))}{\phi(r)}+\log 2.
$$
\end{lemma}

\begin{lemma}\label{maxlem}
Let $f$ be an entire function defined as in  (\ref{def0}). Then there is a set $E$ of finite logarithmic measure such that, for all large $r\notin E$,
$$
\log M(r,f)\leq \log\sigma(r,f) +\log\log\sigma(r,f) +O(1).
$$
\end{lemma}
\begin{proof}
It is clear that, for any $R>r$,
$$
M(r, f)\le \sum_{j=0}^{\infty}|a_j|r^{j}
\le \sigma(R, f)(\frac{R}{R-r})^{1/2}.
$$
Applying Lemma \ref{Borel} to $F(r)=\sigma(r,f)$, $\phi(r)=r$, $\psi(x)=(\log x)^2$ and $\ds R=r+\frac{r}{\psi(F(r))}$ gives
$$\log \sigma(R, f)\le \log \sigma(r, f)+1 \qquad \mbox{and} \qquad
\log  \frac{R}{R-r}\le 2 \log \log \sigma(r, f)+\log 2,
$$
for all large $r \notin E$. It follows the lemma is proved.
\end{proof}
\begin{remark}
It is straightforward to know that $\sigma(r, f)\le M(r, f)$ for all $r>0$.
\end{remark}

\begin{lemma}[{\bf Gol'dberg-Grinshtein}, \cite{gogr1976}]\label{lem4}
Let $f$ be a meromorphic function in the complex plane and let
$0<\alpha<1$. Then there is $r_0$ such that, for any $r_0<r<R$, we have
\begin{align*}
\int_0^{2\pi}\left|\frac{rf'(re^{i\theta})}{f(re^{i\theta})}\right|^\alpha\frac{d\theta}{2\pi} \leq
C(\alpha)\left(\frac{R}{R-r}\right)^{\alpha}T^{\alpha}(R, f) +O(1),
\end{align*}
where $C(\alpha)$ is a constant only depending on $\alpha$.
\end{lemma}

\section{Proofs of our main theorems}\label{S5}
\subsection{Proof of Theorem \ref{thm1}}

Let $f_\omega$ be a random entire function on $(\Omega,\, \mathcal{F}, \, \mu)$ of the form (\ref{def1}).
For $\omega\in\Omega$, by Jensen-Poisson formula,
\begin{align*}
N(r,0, f_\omega)&=\int_0^{2\pi}\log|f_\omega(re^{i\theta})|\frac{d\theta}{2\pi} -\log|c_{f_\omega}(0)|\\
&=\log \sigma(r,f)+\int_0^{2\pi}\log|\hat{f_\omega}(re^{i\theta})|\frac{d\theta}{2\pi} -\log|c_{f_\omega}(0)|.
\end{align*}
It follows that, for any $r>0$,
$$
|N(r,0,f_\omega)-\log\sigma(r, f)+ \log|c_{f_\omega}(0)||
 \le  \int_0^{2\pi}|\log|\hat{f_\omega}(re^{i\theta})| |\frac{d\theta}{2\pi}=X_r. \\
$$
Since $\log \sigma(r,f)$ is increasing, for any positive integer $n$, there is
$r_n$ such that $\log \sigma(r_n, f)=n$ and the sequence $\{r_n\}$ is increasing. Since $f_\omega \in \mathcal{Y}$, there are positive constants $A$ and $B$ such that $\ex(\exp(A|X_r|^B))=C_1<+\infty$. For any $C>1$, set
$$
A_n=\left\{\omega\in\Omega: |N(r_n,0, f_\omega)-\log \sigma(r_n, f)+ \log|c_{f_\omega}(0)||\ge \left(\frac{C}{A}\log n\right)^{\frac1{B}}\right\}.
$$
Therefore, by (\ref{lem2-1}) in Lemma \ref{lem2},
$$\pr(A_n)\le \pr\left(|X_r|\ge \left(\frac{C}{A}\log n\right)^{\frac1{B}}\right)\le \frac{C_1}{n^{C}}.$$
Consequently,
$\ds \sum\pr(A_n)< \infty$  and by the Borel-Cantelli lemma,
$$ \mu(A):=\mu(\cap_{j=1}^{\infty}\cup_{n=j}^{\infty}A_n)=0.$$
For any $\omega \notin A$, there is $j_0=j_0(\omega)$ such that
$\omega \notin \cup_{n=j_0}^{\infty} A_n$, i.e., $\omega\not\in A_n$ for all $n>j_0$. This means that
if $\omega \in \Omega\setminus A$, then, there exist $j_0$ such that for all $n>j_0$, we have
$$
|N(r_n,0, f_\omega)-\log \sigma(r_n, f)+ \log|c_{f_\omega}(0)||< \left(\frac{C}{A}\log n\right)^{\frac1{B}}.
$$

It follows that, for $r\in (r_n, r_{n+1}]$ with $n>j_0$, for almost all $\omega\in\Omega$, we have
\begin{align}\label{thm3.3-1} \nonumber
\log \sigma(r, f) & \le  \log \sigma(r_{n+1}, f)=\log \sigma(r_n, f)+1\\  \nonumber
& \le   N(r_n,0, f_\omega)+(C/A)^{1/B}(\log \log \sigma(r_n, f))^{1/B} +1+\log|c_{f_\omega}(0)| \\
&\le  N(r,0,f_\omega)+(C/A)^{1/B}(\log \log \sigma(r, f))^{1/B}+O_\omega(1),
\end{align}
and
\begin{align}\label{thm3.3-2} \nonumber
 N(r,0, f_\omega)& \le  N(r_{n+1},0, f_\omega)\le  \log \sigma(r_{n+1}, f)+ (C/A)^{1/B}(\log \log \sigma(r_{n+1}, f))^{1/B} +\log|c_{f_\omega}(0)|        \\  \nonumber
& \le \log \sigma(r_n, f)+1+(C/A)^{1/B}(\log \log \sigma(r_n, f))^{1/B}+O_\omega(1)  \\
&\le   \log \sigma(r, f)+(C/A)^{1/B}(\log \log \sigma(r, f))^{1/B}+O_\omega(1),
\end{align}
where $O_\omega(1)$ depends on $A$, $B$ and $C$.

Combining (\ref{thm3.3-1}) and  (\ref{thm3.3-2})  gives
$$
|\log \sigma(r, f)- N(r,0, f_\omega)|\le (C/A)^{1/B}(\log \log \sigma(r,f) )^{1/B}.
$$

Therefore, the theorem is proved completely. \qed

\subsection{Proof of Theorem \ref{newthm}}

To prove Theorem \ref{newthm}, we need the following lemma, whose proof is based on the result of our Theorem \ref{thm1}.

\begin{lemma}\label{lem6}
Let $f_\omega(z) \in \mathcal{Y}$. Then there exists a constant $r_0=r_0(\omega)$ such that, for $r>r_0$, we have
$$
\log\sigma(r, f_\omega)\le \log \sigma(r,f) +\log\log \sigma(r,f) +2 \qquad a.s.
$$
and
$$
\log \sigma(r, f) -(C/A)^{1/B}(\log \log \sigma(r, f))^{1/B}+O_\omega(1)
\le \log\sigma(r, f_\omega),\qquad \mbox{a.s.}
$$
where $C>1$ is any constant and constants $A, B$ are from Condition $Y$.
\end{lemma}
\begin{proof}
Since
$\ds
\ex (\sigma^2(r,f_\omega))=
\sum_{j=0}^{\infty}\ex(|\chi_j(\omega)|^2)|a_n|^2 r^{2n}=\sigma^2(r,f),
$
so, by Lemma \ref{lem1},
$$
\pr\left(\sigma^2(r,f_\omega)>\sigma^2(r,f)\varphi(\sigma(r,f))\right)\leq \frac{\ex (\sigma^2(r,f_\omega))}{\sigma^2(r,f)\varphi(\sigma(r,f))}=\frac1{\varphi(\sigma(r,f))}.
$$
For any positive integer $n$, there is $r_n$ such that $\sigma(r_n, f)=e^n$ and the sequence $\{r_n\}$ is increasing. Set
$$
B_n=\{\omega\in\Omega: \sigma^2(r_n, f_\omega)>\sigma^2(r_n,f)\varphi(\sigma(r_n,f))\}.
$$
Thus, by taking $\varphi(x)=(\log x)^2$, we have
$$
\pr(B_n)\le \frac{1}{\varphi(\sigma(r_n,f))}= \frac{1}{n^2}.
$$
Consequently,
$\ds \sum\pr(B_n)< +\infty$  and
$\ds  \mu(B)\stackrel{def}{=}\mu(\cap_{j=1}^{\infty}\cup_{n=j}^{\infty}B_n)=0$ by
the Borel-Cantelli lemma.
For any $\omega \notin B$, there is $j_1=j_1(\omega)$ such that $\omega \notin \cup_{n=j_0}^{\infty} B_n$, i.e., $\omega\not\in B_n$ for all $n>j_1$. Therefore, for $r\in (r_n, r_{n+1}]$ with $n>j_1$, we have
\begin{align*}
\sigma^2(r,f_\omega)&\leq \sigma^2(r_{n+1},f_\omega)\leq\sigma^2(r_{n+1},f)\varphi(\sigma(r_{n+1},f))\\
&= (e^{n+1})^2(n+1)^2=e^2\sigma^2(r_n,f)(\log \sigma(r_n,f)+1)^2\\
&\leq e^2\sigma^2(r,f)(\log \sigma(r,f)+1)^2.
\end{align*}
It follows that
\begin{align*}\label{e6.1}
\log \sigma(r,f_\omega)\leq \log\sigma(r,f)+\log\log\sigma(r,f)+2  \qquad{a. s.}
\end{align*}

On the other hand, by Theorem \ref{thm1}, Theorem \ref{fmt} and Lemma \ref{lem3}, for  any $C>1$, there is a constant $r_0=r_0(\omega)$, for $r>r_0(\omega)$,
\begin{align*}
\log \sigma(r,f)&\leq N(r,0,f_\omega)+(C/A)^{1/B}(\log \log \sigma(r, f))^{1/B}+O_\omega(1)\\
&\leq T(r,f_\omega)+(C/A)^{1/B}(\log \log \sigma(r, f))^{1/B}+O_\omega(1)\\
&\leq \log \sigma(r,f_\omega)+(C/A)^{1/B}(\log \log \sigma(r, f))^{1/B}+O_\omega(1).
\end{align*}
It follows that the lemma is proved completely.
\end{proof}

We are now ready to prove Theorem \ref{newthm}.
\vskip.1in

\noindent
\begin{proof}[{\it Proof of Theorem \ref{newthm}}]
By applying Theorem \ref{thm1} to $f^*_\omega$, we obtain two positive constants $A$, $B$ such that for any positive constant $C>1$, there exists a constant $r_0=r_0(\omega)>0$, for $r>r_0$,
$$
\log \sigma(r,f^*)\leq N(r, 0, f^*_\omega)+(C/A)^{1/B}(\log\log \sigma(r, f^*) )^{1/B}\qquad a.s.
$$
Since $\sigma(r,f^*)\geq\sigma(r,f)\geq e$ for all large $r$, say, $r>r_0$, and the function $y(x)=x-C_0\log x $ is increasing on $[x_0,+\infty)$, then we have
\begin{equation}\label{thm3.5-1}
\log \sigma(r,f)\leq N(r, 0, f^*_\omega)+(C/A)^{1/B}(\log\log \sigma(r, f) )^{1/B}\qquad a.s.
\end{equation}
By Jensen-Poisson formula, Lemma \ref{lem4} and Theorem \ref{fmt}, we have,
for any $r<R$,
\begin{align*}
N(r, 0, f^*_\omega)-N(r, a, f_\omega)&=\frac1{2\pi}\int_0^{2\pi} \log\left|\frac{f^*_\omega(re^{i\theta})}{f_\omega(re^{i\theta})-a}\right|\,d\theta-\log\frac{|c_{f^*_\omega}(0)|}{|c_{f_\omega}(a)|}\\
&\leq \log T(R,f_\omega)+\log\frac{R}{R-r}+O_\omega(1).
\end{align*}
It follows from Lemma \ref{lem3} that, for any $r<R$,
\begin{equation}\label{thm3.5-2}
N(r, 0, f^*_\omega)\le N(r, a, f_\omega)+\log\log \sigma(R,f_\omega)+\log\frac{R}{R-r}+O_\omega(1).
\end{equation}

Applying Lemma \ref{Borel} to functions $F(r)=\log \sigma(r,f_\omega)$, $\phi(r)=r$, $\psi(r)=\log^2r$ and $\ds R=r+\frac{r}{\psi(F(r))}$, we get
a set $E\subset[r_0,+\infty)$ of finite logarithmic measure, for all large $r$, say, 
$r>r_0$, and $r\not\in E$,
$$
\log \log \sigma(R, f_\omega)<\log  \log \sigma(r,f_\omega)+1
$$
and
$$\log \frac{R}{R-r} \le 2 \log \log\log\sigma(r, f_\omega)+\log 2.
$$

Thus, plugging above two estimates to (\ref{thm3.5-2}) gives
\begin{eqnarray}\label{thm3.5-3}
N(r, 0, f^*_\omega)  \le N(r, a, f_\omega)+\log\log \sigma(r, f_\omega)+2\log \log \log \sigma(r, f_\omega)+ O_\omega(1),
\end{eqnarray}
for $r>r_0$ and $r \notin E$.
It follows from (\ref{thm3.5-1}), (\ref{thm3.5-3}) and Lemma \ref{lem6} that there is $r_1=r_1(a, r_0)$ such that
\begin{equation}\label{thm3.5-4}
\log \sigma(r, f)\le N(r, a, f_\omega)+ (C/A)^{1/B}(\log \log\sigma(r, f))^{1/B}+(1+o(1))\log\log \sigma(r, f),
\end{equation}
for $r>r_1$ and $r \notin E$.

On the other hand, by Nevanlinna's first main theorem, Lemma \ref{lem3} and
Lemma \ref{lem6},
\begin{align*}
N(r, a, f_\omega)& \le T(r, a, f_\omega)=T(r, f_\omega)+O_\omega(1)\\
& \le \log \sigma(r, f_\omega)+O_\omega(1).\\
&\le \log\sigma(r,f)+(1+o(1))\log\log\sigma(r,f).
\end{align*}

Combining this with (\ref{thm3.5-4}) completes the proof of the theorem.
\end{proof}

\subsection{Proof of Corollaries}

\begin{proof}[{\bf Proof of Corollary \ref{cor1}}]
Since the function $x-(C/A\log x)^{1/B}$ is a increasing function for all large $x$
and Lemma \ref{lem3},
we have
$$
T(r,f)-(C/A)^{1/B}(\log T(r, f))^{1/B} \le
\log \sigma(r, f)-(C/A)^{1/B}(\log \log \sigma(r, f))^{1/B}.
$$
The rest of proof is a straightforward consequence of Theorem \ref{thm1}.
\end{proof}

\begin{proof}[{\bf Proof of Corollary \ref{cor2}}]
Let $f_\omega$ be a Gaussian entire function. By (\ref{lem4.4-1}) in the proof of Lemma \ref{lem2}, we have, for any $\tau>0$ and for any $r>0$,
$$
\ex(e^{\frac2{1+\tau}|X_r|})< \infty.
$$
So, $f_\omega \in \mathcal{Y}$ by choosing $A=2/(1+\tau)$ and $B=1$. So applying Corollary \ref{cor1} for $A=2/(1+\tau)$,  $B=1$ and $C=1+\tau$, we can finish the proof of the corollary in this case.

Let $f_\omega$ be a Steinhaus entire function. By (\ref{lem4.4-2}) in the proof of Lemma \ref{lem2}, we have, for any $\tau>0$ and for any $r>0$,
$$
\ex(e^{\frac1{1+\tau}|X_r|})< \infty.
$$
So, $f_\omega \in \mathcal{Y}$ by choosing $A=1/(1+\tau)$ and $B=1$. So applying Corollary \ref{cor1} for $A=1/(1+\tau)$,  $B=1$ and $C=1+\tau$, we can finish the proof of the corollary in this case.

Let $f_\omega$ be a Rademacher entire function. By (\ref{lem4.4-3}) in the proof of Lemma \ref{lem2}, we have, for any $\tau>0$ and for any $r>0$,
$$
\ex(e^{\frac6{eC_0} |X_r|^{1/6}})< \infty.
$$
So, $f_\omega \in \mathcal{Y}$ by choosing $A=6/{eC_0}$ and $B=1/6$. So applying Corollary \ref{cor1} for the $A$ and $B=1/6$ and $C=1+\epsilon$, we can finish the proof of the corollary in this case.
\end{proof}

\section*{Acknowledgments}

The authors would like to thank Professors P. V. Filevych,  I. Laine,  A. Nishry and M. Sodin for their help during the preparation of the manuscript.

\noindent The first named author would like to thank the University of North Carolina Wilmington for its hospitality during her visit from 2019 to 2020 and was supported by China Scholarship Council (No. 201906470026). The second named author was supported by National Natural Science Foundation of China (No. 11771090). And the third named author was supported by National Natural Science Foundation of China (No. 11901311).


\begin{thebibliography}{99}
\bibliographystyle{ijmart}


\bibitem{buGlSo2019}
L. Buhovsky, A. Gl\"ucksam, M. Sodin, Translaton invariant probability measures on entire functions, J. Anal Math., 139, (2019), no. 1, 307-339.

\bibitem{chYe01}
W. Cherry, Z. Ye,  Nevanlinna's theory of value distribution, Springer, 2001.

\bibitem{gogr1976}
A.A. Gol'dberg, A. Grinshtein, The logarithmic derivative of a meromorphic function, Mat. Zametki, 19, (1976), 525-530.

\bibitem{Hayman64}
W.K. Hayman, Meromorphic Functions, Clarendon Press, 1964.

\bibitem{hoKrPeVi09}
B.J. Hough, M. Krishnapur, Y. Peres and B. Vir\'{a}g,  Zeros of Gaussian analytic functions and determinantal point processes. University Lecture Series, 51. American Mathematical Society, Providence, RI, 2009.

\bibitem{kabZap14}
Z. Kabluchko and D. Zaporozhets, Asymptotic Distribution of Complex Zeros of Random Analytic Functions, Ann. of Probability, 42, (2014), no. 4, 1374-1395.

\bibitem{Kahane85}
J.P. Kahane, Some random series of functions. 2nd edition. Cambridge Univ. Press, 1985.

\bibitem{liOf45}
J.E. Littlewood, A.C. Offord, On the distribtuion of the zeros and $ a$-value of a random integral function. I.,
J. London Math. Soc. 20, (1945), 30-136.

\bibitem{LO48}
J.E. Littlewood, A.C. Offord, On the distribution of the zeros and a-values of a random integral function. II.,
Ann. Math. 49 (1948), 885-952. Errata, 50 (1949), 990-991.

\bibitem{Mahola13}
M.P. Mahola, The value distribution of random analytic functions, thesis, 2013.

\bibitem{maFi10}
M.P. Mahola, P.V. Filevych,
The value distribution of a random entire function,
Matematychni Studii, 34, (2010), no. 2, 120-128.

\bibitem{MF12}
M.P. Mahola, P.V. Filevych,
The angular value distribution of random analytic functions,
Matematychni Studii, 37, (2012), no. 1, 34-51.

\bibitem{mu82}
T. Murai, The value-distribution of random entire functions, Kodai Math J. 5, (1982), 313-317.

\bibitem{NNS14}
F. Nazarov, A. Nishry, M. Sodin, Log-integrability of Rademacher Fourier series, with applications to
random analytic functions, St. Petersburg Mathematical Journal, 25, (2013), no. 3, 467-494.

\bibitem{nns2016}
F. Nazarov, A. Nishry, M. Sodin, Distribution of zeroes of Rademacher Taylor series, Ann. de la Facult\'e des Sci. de Tou, 25, (2016), no. 4, 759-4784.

\bibitem{offord72}
A.C. Offord,
The distribution of the values of a random function in the unit disk,
Studia Math. 41, (1972), 71-106.

\bibitem{sodin05}
M. Sodin, Zeroes of Gaussian analytic functions, European Congress of Mathematics, Eur. Math. Soc., (2005), 445-458.

\bibitem{SC00}
D.C. Sun, T.W. Chen, Deficient functions of the random function,
Progr. Natur. Sci. (English Ed.) 10, (2000), no. 3, 177-182.

\bibitem{sun90}
D.C. Sun, Q.S. Liu, On the value-distribution of some random analytic functions. J. of Math. 3, (1990), 285-298.

\bibitem{ye95}
Z. Ye, On Nevanlinna's second main theorem in projective space, Invent. Math. 122, (1995), 475-507.

\end{thebibliography}
\end{document}